\documentclass[a4paper, 11pt]{article}

\usepackage[
            includefoot,  
            marginparwidth=0in,     
            marginparsep=0in,       
            margin=1.45in,               
            includemp]{geometry}

\usepackage{bbm}
\usepackage{bigints}
\usepackage{enumitem}
\usepackage{float}
\usepackage{graphicx}                  
\usepackage{amssymb}
\usepackage{amsfonts}
\RequirePackage{amsmath}
\RequirePackage{amsthm}

\usepackage{color, soul}
\usepackage{hyperref}
\hypersetup{
	colorlinks   = true,
	citecolor    = black,
	linkcolor    = blue
}
\usepackage{fancyhdr}

\newtheorem{thm}{Theorem}[section]
\newtheorem{defn}[thm]{Definition}
\newtheorem{prop}[thm]{Proposition}
\newtheorem{lem}[thm]{Lemma}

\newtheorem{rmq}{Remark}[section]

\DeclareMathOperator{\Hess}{\nabla^2}

\DeclareMathOperator{\Id}{I_d}
\DeclareMathOperator{\Tr}{Tr}

\newcommand{\R}{\mathbb{R}}

\newcommand{\E}{\mathbb{E}}
\newcommand{\PP}{\mathbb{P}}

\begin{document}

\title{Stability estimates for invariant measures of diffusion processes, with applications to stability of moment measures and Stein kernels}
\author{Max Fathi and Dan Mikulincer}
\date{\today}

\maketitle

\begin{abstract}
We investigate stability of invariant measures of diffusion processes with respect to $L^p$ distances on the coefficients, under an assumption of log-concavity. The method is a variant of a technique introduced by Crippa and De Lellis to study transport equations. As an application, we prove a partial extension of an inequality of Ledoux, Nourdin and Peccati relating transport distances and Stein discrepancies to a non-Gaussian setting via the moment map construction of Stein kernels. 
\end{abstract}

\section{Introduction}
Let $X_t, Y_t$ be stochastic processes in $\R^d$ which satisfy the SDEs,
\begin{equation} \label{eq_processes}
dX_t = a(X_t)dt + \sqrt{2\tau(X_t)}dB_t,\ \ dY_t = b(Y_t)dt + \sqrt{2\sigma(Y_t)}dB_t.
\end{equation}
Here $B_t$ is a standard Brownian motion, $a,b$ are vector-valued functions, and $\sigma,\tau$ take values in the cone of symmetric $d \times d$ positive definite matrices, which we shall denote by $\mathcal{S}_d^{++}$. 
Given $X_0$ and $Y_0$, we shall write the marginal laws of the processes as $X_t \sim \mu_t$ and $Y_t \sim \nu_t$. \\

Suppose that, in some sense to be made precise later, $a$ is close to $b$, and $\tau$ is close $\sigma$. One can ask whether the measure $\mu_t$ must then be close to $\nu_t$. Our goal here is to study the quantitative regime of this problem. The method used here is an adaptation of a technique developed by Crippa and De Lellis for transport equations. That method was introduced in the SDE setting in \cite{CJ18, LL18}. Our implementation here will be a bit different, to allow for estimates in weighted Sobolev space that behave better for large times, and will allow us to compare the invariant measures of the two processes, under suitable ergodic assumptions. \\

We will be especially interested in the case where $\sigma$ and $\tau$ are uniformly bounded from below. In such a situation, $X_t$ and $Y_t$ admit unique invariant measures \cite{BD17}, which we shall denote, respectively, as $\mu$ and $\nu$. In this setting, we will think about $\nu$ as the reference measure and quantify the discrepancy in the coefficients as:
$$ \beta := \|a-b\|_{L^1(\nu)} + \|\sqrt{\sigma}-\sqrt{\tau}\|_{L^2(\nu)}.$$ 
Our main result is an estimate of the form (see precise formulation below)
$$\mathrm{dist}(\mu,\nu) \leq h(\beta),$$
where $\mathrm{dist}(\cdot,\cdot)$ stands for an appropriate notion of distance, which will here be a transport distance, and $\lim\limits_{\beta\to 0} h(\beta)= 0.$\\

As an application of our stability estimate, we will show that if two uniformly log-concave measures satisfy certain similar integration by parts formula, in the sense arising in Stein's method, then the measures must be close. This problem was the original motivation of our study. 

\subsection{Background on stability for transport equations}

Part of the present work is a variant in the SDE setting of a now well-established quantitative theory for transport equations with non-smooth coefficients, pioneered by Crippa and De Lellis \cite{CDL}. 

As demonstrated by the DiPerna-Lions theory (\cite{DPL89} and later estimates in \cite{CDL}), there is a significant difference in the stability of solutions to differential equations when the coefficients are Lipschitz continuous versus when they only belong to some Sobolev space (and are not necessarily globally Lipschitz). Our focus will be on the latter, and arguably more challenging, case. As we shall later discuss, the techniques can be carried over to the setting of stochastic differential equations we are interested in here, as worked out in \cite{CJ18, LL18}.

The strategy for quantitative estimates, introduced in \cite{CDL}, in the Lagrangian setting relies on controlling the behavior of $\ln(1 + |X_t - Y_t|/\delta)$ for two flows $(X_t)$ and $(Y_t)$, with a parameter $\delta$ that will be very small, of the order of the difference between the vector fields driving the flows. A crucial idea is the use of the Lusin-Lipschitz property of Sobolev vector fields \cite{Liu77}, which allows to get Lipschitz-like estimates on large regions, in a controlled way. We will discuss this idea in more details in Section \ref{sect_ll}. The ideas of \cite{CDL} were adapted to the Eulerian setting in \cite{Sei17, Sei18}, using a transport distance with logarithmic cost. 

Existence and uniqueness of solutions to SDEs and Fokker-Planck equations with non-smooth coefficients by adapting DiPerna-Lions theory was first addressed in \cite{LBL08, Fi08}, inspiring many further developments, such as \cite{FLT10, Tre16, XZ16, CJ18}. We will not discuss much the issue of well-posedness here, and focus on more quantitative aspects of the problem. We refer for example to \cite[Section 3.1]{LL18} and \cite{Tre16} for a comprehensive discussion of the issues, in particular with respect to the different ways of defining a notion of solution. As pointed out in \cite{CJ18, LL18}, the kind of quantitative methods used here could also prove well-posedness by approximation with processes with smoother coefficients.

\section{Results}
\subsection{Main result}
We consider two diffusion processes of the form \eqref{eq_processes}, and assume that they admit unique invariant measures $\mu$ and $\nu$. We fix some real number $ p \geq 2$ with $q$ its H\"older conjugate, so that $\frac{1}{q} + \frac{1}{p} = 1$. We then make the following assumptions:
\begin{enumerate}[label={H\arabic*.}, ref=H\arabic*]
	\item \label{ass:coeff}(Regularity of coefficients) There exists a function $g: \R^d \to \R$, such that for any $x,y \in \R^d$:
	\begin{align} \label{eq_sob_to_lip}
	\|a(x) - a(y)\|,\|\sqrt{\tau(x)} - \sqrt{\tau(y)}\| \leq (g(x) + g(y))\|x-y\|,
	\end{align}
	and
	$$\|g\|_{L^{2q}(\mu)} <\infty.$$
	To simplify some notations later on, we also assume that $g \geq 1$ pointwise, which does not strengthen the assumption since we can always replace $g$ by $\max(g, 1)$. 
	\item \label{ass:density} (Integrability of relative density) Both $\mu$ and $\nu$ have finite second moments, and it holds that
	$$\left\|\frac{d\nu}{d\mu}\right\|_{L^p(\mu)}< \infty.$$
	\item \label{ass:wass} (Exponential convergence to equilibrium) There exist constants $\kappa,C_H > 0$ such that for any initial data $\mu_0$ and any $t \geq 0$ we have
	$$W_2(\mu_t,\mu)\leq C_He^{-\kappa\cdot t}W_2(\mu_0,\mu).$$

\end{enumerate}
We denote for the second moments 
$$m^2_2(\mu) = \int\limits_{\R^d}|x|^2d\mu,\  m^2_2(\nu) = \int\limits_{\R^d}|x|^2d\nu.$$

Concerning Assumption \eqref{ass:coeff}, it is the Lusin-Lipschitz property we mentioned above, and which we shall discuss in some depth in Section \ref{sect_ll} below. One should think about it as being a generalization of Lipschitz continuity, as it essentially means $\sigma$ and $a$ may be well approximated by Lipschitz functions on arbitrarily large sets. In particular, this property holds for Sobolev functions when the reference measure is log-concave.   
When $C_H = 1$, the third assumption corresponds to contractivity, which for reversible diffusion processes is equivalent to a lower bound on the Bakry-Emery curvature \cite{vRS05}. Allowing for a constant $C_H > 1$ allows to cover other examples, such as hypocoercive dynamics, notably because the assumption is then invariant by change of equivalent metric, up to the value of $C_H$. See for example \cite{Mon19}

Now, for $R > 0$, define the truncated quadratic Wasserstein distance by:
\begin{equation}
\widetilde{W}^2_{2,R}(\nu, \mu) := \inf_{\pi} \int{\min(|x-y|^2, R)d\pi},
\end{equation}
where the infimum is taken over all couplings of $\mu$ and $\nu$.
This is a distance on the space of probability measures, weaker than the classical $W_2$, defined by 
\begin{equation*}
W^2_{2}(\nu, \mu) := \inf_{\pi} \int{|x-y|^2d\pi}.
\end{equation*}
With the above notations our main result reads:
\begin{thm} \label{thm_main_stability}
	Assume \eqref{ass:coeff},\eqref{ass:density} and \eqref{ass:wass} hold, and denote 
	$$ \beta := \|a-b\|_{L^1(\nu)} + \|\sqrt{\sigma}-\sqrt{\tau}\|_{L^2(\nu)}.$$ 
	Then, for any $R > 1$,
	$$\widetilde{W}_{2,R}^2(\nu,\mu) \leq 100C_H^2R\cdot \|g\|^2_{L^{2q}(\mu)}\left\|d\nu/d\mu\right\|_{L^p(\mu)}\frac{\ln\left(\ln\left(1 + \frac{R}{\beta}\right)\right) +\ln\Big(m_2(\mu) + m_2(\nu)\Big)+\kappa\cdot R}{\kappa\cdot \ln\left(1 + \frac{R}{\beta}\right) },$$
	where $\frac{1}{q} + \frac{1}{p} = 1$.
\end{thm}
\begin{rmq}\label{rm_w1_w2}
	Essentially, The theorem says that $\widetilde{W}_{2,R}(\nu,\mu)$ decreases at a rate which is proportional to $\sqrt{\ln\left(1 + \frac{R}{\beta}\right)^{-1}}$. We could improve the rate to $\ln\left(1 + \frac{R}{\beta}\right)^{-1}$ by considering a truncated $W_1$ distance, as shall be discussed in Remark \ref{remark_w1_again}. However, for our application to Stein kernels it is more natural to work with $W_2$.
\end{rmq}
Let us discuss now the role of the term $\left\|d\nu/d\mu\right\|_{L^p(\mu)}$ in Theorem \ref{thm_main_stability}. In order to use $\|a-b\|_{L^1(\nu)} + \|\sqrt{\sigma}-\sqrt{\tau}\|_{L^2(\nu)}$ as a measure of discrepancy, it seems necessary that the supports of $\mu$ and $\nu$ intersect, otherwise we could just change $\sigma$ and $a$ on a $\mu$-negligible set and have $\beta = 0$ in the conclusion of the theorem, which obviously fails in this particular situation. Thus, since the bound in Theorem \ref{thm_main_stability} only makes sense when $\left\|d\nu/d\mu\right\|_{L^p(\mu)}$ is finite, one may view this term as an a-priori guarantee on the common support of $\mu$ and $\nu$.\\

The logarithmic rate obtained here may seem quite weak. In the setting of transport equations, the logarithmic bounds obtained by the method considered here are sometimes actually sharp \cite{Sei18}. We do not know much about optimality in the stochastic setting, since it may be that the presence of noise would help, while the method used here cannot do better in the stochastic setting than in the deterministic setting.\\ 

As alluded to in the introduction, if the coefficients $a$ and $\sqrt{\tau}$ are actually $L$-Lipschitz, in which case $g\equiv \frac{L}{2}$ in \eqref{ass:coeff}, then one may greatly improve the rate in Theorem \ref{thm_main_stability}.
\begin{thm} \label{thm_stability_lip}
	Assume $a$ and $\sqrt{\tau}$ are $L$-Lipschitz and that \eqref{ass:wass} holds, and denote 
	$$ \beta := \|a-b\|_{L^2(\nu)} + \|\sqrt{\tau}-\sqrt{\sigma}\|_{L^2(\nu)}.$$ 
	Then, 
	$$W_{2}(\nu,\mu) \leq 15C_H^{\frac{4L^2+1}{2\kappa}}\beta\left(\frac{L}{\kappa} +1\right).$$
\end{thm}

This type of estimate is part of the folklore, and a version of it appears for example in \cite{BRS}. 

\subsection{About the Lusin-Lipschitz property for Sobolev functions}
\label{sect_ll}

We will now discuss in some more depth Assumption \eqref{ass:coeff}. As mentioned previously, it is motivated by the Lusin-Lipschitz property of Sobolev functions with respect to the Lebesgue measure \cite{Liu77}: if a function $f : \mathbb{R}^d \longrightarrow \R$ satisfies $\int{|\nabla f|^pdx} < \infty$ then for a.e. $x, y \in \R^d$ we have
\begin{equation} \label{hl_bnd}
|f(x) - f(y)| \leq (M|\nabla f|(x) + M|\nabla f|(y))|x-y|,
\end{equation}
where $M$ is the Hardy-Littlewood maximal operator, defined on a non-negative function $g$ as,
$$Mg(x) := \sup_{r > 0} |B_r|^{-1}\int_{B_r(x)}{g(x)dx},$$
$C$ a dimension-free constant, and $B_r$ is the Euclidean ball of radius $r$. This operator satisfies the dimension-free continuity property
$$||Mf||_{L^p(dx)} \leq C_p||f||_{L^p(dx)},$$
when $p > 1$. The dimension-free bound on $C_p$ is due to E. Stein \cite{SS83}

In particular, if $\nabla f \in L^p(dx)$, then $f$ is $\lambda$-Lipschitz on the regions where $M|\nabla f|$ is smaller than $\lambda/2$, which are large when $\lambda$ is, by the Markov inequality. The important distinction between using an estimate on $M|\nabla f|$ instead of $\nabla f$ is that, even if both $\nabla f(x)$ and $\nabla f(y)$ are controlled, we do not automatically get an estimate on $f(x)-f(y)$, since the straight line from $x$ to $y$ may well go through a region where $|\nabla f|$ is arbitrarily large. The use of \eqref{hl_bnd} nicely bypasses this issue. 

An important issue for the applications we shall discuss here is that working in functional spaces weighted with the Lebesgue measure is not always the most natural when dealing with stochastic processes. It is often preferable to work in $L^p(\mu)$ with a probability measure adapted to the problem considered, which here shall be the invariant measure of the stochastic process considered. However, in general the maximal operator has no reason to be continuous over $L^p(\mu)$, unless $\mu$ has density with respect to the Lebesgue measure that is uniformly bounded from above and below on its support. As soon as $\mu$ is not compactly supported, this cannot be the case. Therefore, we shall make strong use of a work of Ambrosio, Bru\'e and Trevisan \cite{ABT}, which proves a Lusin-type property for Sobolev functions with respect to a log-concave measure. The proof uses an operator different from the Hardy-Littlewood maximal operator, more adapted to the setting. We shall not discuss here the specifics of that operator, since we only need the Lusin property, and not the maximal operator itself. The exact statement of their result, in the restricted setting of log-concave measures on $\R^d$, is as follows: 

\begin{prop} \label{prop_RCD}
Let $\mu$ be a log-concave measure, and $p \geq 2$. Then for any function $f \in W^{1,p}(\mu)$, there exists a function $g$ such that
$|f(x)-f(y)| \leq (g(x) + g(y))|x-y|$ for a.e. $x$ and $y$, and with $||g||_{L^p(\mu)} \leq C_p||\nabla f||_{L^p(\mu)}$ with $C_p$ some universal constant, that only depends on $p$. 
\end{prop}

This statement is proved in \cite[Theorem 4.1]{ABT}, and also holds for maps taking values in some Hilbert space. It is written there only for $p=2$, but the reason for that restriction is that they work in the more general setting of possibly nonsmooth RCD spaces, rather than just $\R^d$ endowed with a measure. The only point where they require the restriction to $p=2$ is when using the Riesz inequality \cite[Remark 3.9]{ABT}, which in the smooth setting is known for general values of $p$, as proved in \cite{Bak87}.

\subsection{Related works}

As mentioned previously, the adaptation of the Crippa-De Lellis method to derive quantitative estimates for stochastic differential equations was already considered in \cite{CJ18} and  \cite{LL18}. 

The results of \cite{LL18} give stability estimates with bounds that depend on $||\nabla \sigma||_{L^p(dx)}$. Considering estimates weighted with the Lebesgue measure allows to use the Hardy-Littlewood maximal function directly. As mentioned above, the main focus here is to get estimates that are weighted with respect to a probability measure adapted to the problem, which may behave very differently, for example when the coefficients of the two SDE are uniformly close, but not compactly supported. 

To use estimates in weighted space, \cite{CJ18} considers functions such that
$$\int{(M|f|)^2d\mu_t} + \int{(M|\nabla f|)^2d\mu_t} < \infty,$$
with $\mu_t$ the flow of the SDE. The authors also consider other function spaces of the same nature, sharper in dimension one, or that handle weaker integrability conditions on $M|\nabla f|$ than $L^p$ (but stronger than $L^1$). Since the space depends on the law of the flow at all times, it may be difficult to determine estimates on such norms. For the application considered in Section \ref{sect_intro_stein}, we do not know whether the approach of \cite{CJ18} could apply. 

We shall also focus on establishing very explicit quantitative estimates in transport distance, highlighting in particular the dependence on the dimension. 

Another approach was developed in \cite{BRS} to directly obtain relative entropy estimates between the distributions at finite times. The upside of that approach is that the quantitative estimates are quite stronger, depending polynomially on some distance between the coefficients. The two downsides are that they depend on stronger Sobolev norms, requiring that the derivatives of the two diffusion coefficients are close in some sense, as well as a-priori Fisher information-like bounds on the relative densities, rather than $L^p$ bounds. Fisher information-like estimates were then derived in \cite{BRS18} by directly comparing generators via a Poisson equation, also using stronger Sobolev norms. 

Finally, when the two diffusion coefficients match, one can derive relative entropy bounds via Girsanov's theorem. Unfortunately, this strategy cannot work when the two diffusion coefficients differ. 

\subsection{An application to Stein kernels} \label{sect_intro_stein}
We now explain how our result might be applied in the context of Stein's method for bounding distances between probability measures. The theory was developed by C. Stein in \cite{Ste72, Ste86} to control distances to the standard Gaussian along the central limit theorem. Since then it has found many applications for bounding distances between probability measures, in both Gaussian and non-Gaussian situations \cite{Ros11, NP12, Cha14}. At the heart of the theory lies the following observation, sometimes called Stein's lemma (see \cite{Ros11}):

 If $G \sim \gamma$ is the standard Gaussian in $\R^d$ then it satisfies the following integration by parts formula, for any regular test function $f$,
$$\E\left[\langle \nabla f(G),G\rangle\right] = \E\left[\Delta f(G)\right],$$
where $\Delta$ is the Laplacian. Moreover, the Gaussian is the only measure which satisfies this formula. 

Given a measure $\nu$ on $\R^d$ and $X \sim \nu$, a matrix valued map $\tau:\R^d \to M^d(\R)$, is said to be a Stein kernel for $\nu$, if it mimics the above formula:
\begin{equation} \label{eq_stein_kernel}
\E\left[\langle \nabla f(X),X\rangle\right] = \E\left[\langle \Hess f(X),\tau(X)\rangle_{HS}\right].
\end{equation}
Observe that the map $\tau_\gamma \equiv \mathrm{I_d}$, which is constantly identity, is a Stein kernel for $\gamma$. Stein's lemma suggests that if $\tau_\nu$ is close to the identity then $\nu$ should be close $\gamma$. This is in fact true, and there are many examples of precise quantitative statements implementing this idea, for various distances between measures, such as transport distances, the total variation distance, or the Kolmogorov distance in dimension 1. The one most relevant to the present work  is an inequality of \cite{LNP15}, which states that for any $\tau$ which is a Stein kernel for a measure $\nu$,
\begin{align} \label{eq_WSH}
W_2^2(\nu,\gamma) \leq \|\tau-\mathrm{I_d}\|_{L^2(\nu)},
\end{align}
where $W_2$ is the quadratic Wasserstein distance. The proof of this inequality strongly relies on Gaussian algebraic identities, such as the Mehler formula for the Ornstein-Uhlenbeck semigroup. We are interested in similar estimates when neither of the measures are Gaussian. The main motivation comes from the fact that Stein's method, in its classical implementations, is hard to use for target measures that do not satisfy certain exact algebraic properties (typically, explicit knowledge of the eigenvectors of an associated Markov semigroup). We shall prove a weaker inequality holds for certain non-Gaussian reference measures and for one particular construction of Stein kernels. To understand this construction we require the following definition.

\begin{defn}[Moment map]
	Let $\mu$ be a measure on $\R^d$. A moment map of $\mu$ is a convex function $\varphi:\R^d \to \R$ such
	that $e^{-\varphi}$ is a centered probability density whose push-forward by $\nabla \varphi$ is $\mu$.
\end{defn}
As was shown in \cite{CEK15, San16}, if $\mu$ is centered and has a finite first moment and a density,
then its moment map exists and is unique as long as we enforce essential continuity at the boundary of its support.
The moment map $\varphi$ can be realized as the optimal transport map between some source log-concave measure and the target measure $\mu$, where we enforce that gradient of the source measure's potential must equal the transport map itself. The correspondence between the convex function $\varphi$ and the measure $\mu$ is actually a bijection, up to a translation of $\varphi$, and the measure associated with a given convex function is known as its moment measure. 

If $\mu$ has a density $\rho$, then $\varphi$ solves the Monge-Amp\`ere-type PDE
$$e^{-\varphi} = \rho(\nabla \varphi)\det(\nabla^2 \varphi).$$
This PDE, sometimes called the toric K\"ahler-Einstein PDE, first appeared in the geometry literature \cite{WZ04, Don08, BB13, Leg16}, where it plays a role in the construction of K\"ahler-Einstein metrics on certain complex manifolds. Variants with different nonlinearities have recently been considered, for example in \cite{HS20}. 

The connection between moment maps and Stein kernels was made in \cite{Fat18}. Specifically, it was proven that if $\varphi$ is the moment map of $\mu$, then (up to regularity issues) the matrix valued map,
\begin{align} \label{eq_moment_kernel}
\tau_\mu := \Hess\varphi (\nabla \varphi ^{-1}),
\end{align}
is a Stein kernel for $\mu$. Since $\varphi$ is a convex function, $\tau_\mu$ turns out to be supported on positive semi-definite matrices.
For this specific construction of a Stein kernel we will prove the following analogue of \eqref{eq_WSH}.
\begin{thm} \label{thm_stein}
	Let $\mu$ be a log-concave measure on $\R^d$ such that
	$$\alpha \mathrm{I_d} \leq -\nabla^2\ln\left(d\mu/dx\right) \leq \frac{1}{\alpha}\mathrm{I_d},$$
	for some $\alpha \in (0, 1]$ and let $\tau_\mu$ be its Stein kernel defined in \eqref{eq_moment_kernel}. If $\nu$ is any other probability measure and $\sigma$ is a Stein kernel for $\nu$ such that \eqref{eq_processes} is well defined, then for $\beta = \|\sqrt{\tau} - \sqrt{\sigma}\|_{L^2(\mu)}$ and $M = \max\left(m_2^2(\mu),m_2^2(\nu)\right)$,
	$$W_2^2(\mu,\nu) \leq C\alpha^{-6}d^{3}M\ln(M)\|d\nu/d\mu\|_\infty \frac{\ln\left(\ln\left(1+\frac{M}{\beta}\right)\right) + M\ln(M)}{\ln\left(1+\frac{M}{\beta}\right)}.$$
	Moreover, if $\mu$ is radially symmetric, has full support, and $d >c$, for some universal constant $c > 0$,
	$$W_2^2(\mu,\nu) \leq C\alpha^{-20}d^{7/2}M\ln(M)\|d\nu/d\mu\|_{L^{2}(\mu)} \frac{\ln\left(\ln\left(1+\frac{M}{\beta}\right)\right) + M\ln(M)}{\ln\left(1+\frac{M}{\beta}\right)}.$$
	Finally, if $\sqrt{\tau_\mu}$ is $L$-Lipschitz, then
	$$W_2^2(\mu,\nu) \leq 100 \alpha^{-(4L^2+1)}\left(2L+1\right)\beta.$$
\end{thm}

It should be emphasized that, except in dimension one, Stein kernels are \emph{not} unique. Different constructions than the one studied here have been provided for example in \cite{Cha09, CFP17, MRS18, NPR10}. Unlike the functional inequalities of \cite{LNP15} for the Gaussian measure, our results will only work for the Stein kernels constructed from moment maps (at least for one of the two measures). In particular, in order to define a stochastic flow from a Stein kernel, we must require the kernel to take positive values, which to our knowledge is not guaranteed for other constructions. 

While this estimate is somewhat weak, it seems to be one of the few instances where we can estimate a distance from a discrepancy for a class of target measures, without explicit algebraic requirements for an associated Markov generator. Recently, there has been progress on implementing Stein's method for wide classes of target measures via Malliavin calculus \cite{FSX19, GDVM19}. 

Note that if one of the two measures is Gaussian, since the natural Stein kernel for the standard Gaussian is constant, and hence Lipschitz, one could use the stronger Theorem \ref{thm_stability_lip} to get a stability estimate, which would still be weaker than that of \cite{LNP15}, but with the sharp exponent.

One may wonder why we do not prove this type of estimate directly using Stein's method. The key difference lies in that we do not need a second-order regularity bound on solutions of Stein's equation, which we do not even know how to prove. To be more precise, the natural way to try to use Stein's method for this problem would be to apply the generator approach using the generator of the process $dX_t = -X_tdt + \sqrt{2\tau(X_t)}dB_t$, where $\tau$ is the Stein kernel for $\mu$. Applying Stein's method to bound say the $W_1$ distance would require us to bound $||\nabla f||_{\infty}$ and $||\nabla^2 f||_{\infty}$ for solutions to the Stein equation
$$-x \cdot \nabla f + \operatorname{Tr}(\tau \nabla^2 f) = g -\int{gd\mu}$$
for arbitrary $1$-Lipschitz data $g$. While a slightly stronger version of Assumption \eqref{ass:wass} could be used to prove bounds on $||\nabla f||_{\infty}$, the techniques used here would not help to bound $||\nabla^2 f||_{\infty}$. So using Stein's method would require some ingredients we do not have. Indeed, in general proving second-order bounds is usually the most difficult step in implementing Stein's method via diffusion processes, and in the literature has mostly been done for measures satisfying certain algebraic properties, such as having an explicit orthogonal basis of polynomials that are eigenvectors for an associated diffusion process (for example Gaussians or gamma distributions). 

\section{Proofs of stability bounds}
A rough outline of the proofs is as follows: as a first step we will use It\^o's formula to show that \eqref{ass:coeff} implies bounds on the measures $\mu_t$ and $\nu_t$, for fixed $t$. Indeed, \eqref{ass:coeff} will allow us to replace quantities like $\|\tau(X_t)-\tau(Y_t)\|$, which will arise through the use of It\^o's formula by something more similar to $\|X_t-Y_t\|$. We will then use \eqref{ass:density} to transfer those estimate to the measure $\nu$ as well.\\
After establishing that $\mu_t$ and $\nu_t$ are close, \eqref{ass:wass} will be used to establish the same for $\mu$ and $\nu$.\\

We first demonstrate this in the easier case of globally Lipschitz coefficients. 

\subsection{Lipschitz coefficients - proof of Theorem \ref{thm_stability_lip}}
	\begin{proof}[Proof of Theorem \ref{thm_stability_lip}]
		By It\^o's formula, we have
		$$d\|X_t - Y_t\|^2 = 2\langle X_t-Y_t, a(X_t) -b(Y_t)\rangle dt + \sqrt{8}(X_t-Y_t)(\sqrt{\sigma}(X_t)-\sqrt{\tau}(Y_t))dB_t + \|\sqrt{\sigma}(X_t)-\sqrt{\tau}(Y_t)\|_{HS}^2dt.$$
		So,
		\begin{align*}
		\frac{d}{dt}\E\left[\|X_t - Y_t\|^2\right] \leq \E\left[\|X_t - Y_t\|^2\right] + \E\left[\|a(X_t)-b(Y_t)\|^2\right] + \E\left[\|\sqrt{\sigma}(X_t)-\sqrt{\tau}(Y_t)\|_{HS}^2\right].
		\end{align*}
		We have
		\begin{align*}
		\E\left[\|a(X_t)-b(Y_t)\|^2\right] &\leq 2\E\left[\|a(X_t)- a(Y_t)\|^2\right] + 2\E\left[\|a(Y_t)- b(Y_t)\|^2\right]\\
		&\leq 2L^2\E\left[\|X_t-Y_t\|^2\right] +2\|a-b\|_{L^2(\nu_s)}^2,
		\end{align*}
		and
		\begin{align*}
		\E\left[\|\sqrt{\sigma}(X_t)-\sqrt{\tau}(Y_t)\|_{HS}^2\right] &\leq 2\E\left[\|\sqrt{\sigma}(X_t)-\sqrt{\sigma}(Y_t)\|_{HS}^2\right] + 2\E\left[\|\sqrt{\sigma}(Y_t)- \sqrt{\tau}(Y_t)\|^2\right]\\
		&\leq 2L^2\E\left[\|X_t-Y_t\|^2\right] +2\|\sqrt{\sigma}-\sqrt{\tau}\|_{L^2(\nu_s)}^2.
		\end{align*}
		Combine the above displays to obtain,
		$$\frac{d}{dt}\E\left[\|X_t - Y_t\|^2\right] \leq (1+4L^2)\E\left[\|X_t - Y_t\|^2\right] +2\|a-b\|_{L^2(\nu_s)}^2 + 2\|\sqrt{\sigma}-\sqrt{\tau}\|_{L^2(\nu_s)}^2.$$
		We choose $\mu_0 = \nu_0 = \nu$ so that $\nu_s = \nu$ for all $s \geq 0$, and denote $r = 2\|a-b\|_{L^2(\nu)}^2 + 2\|\sqrt{\sigma}-\sqrt{\tau}\|_{L^2(\nu)}^2.$
		To bound $W_2^2(\nu,\mu)$, we consider the differential equation
		$$f'(t) = (1+4L^2)f(t) +r, \text{ with initial condition } f(0) = 0.$$
		Its unique solution is given by $f(t) = r \frac{e^{(4L^2 + 1)t}-1}{4L^2 + 1}$.
		Thus, by Gronwall's inequality
		$$W^2_2(\nu,\mu_t) = W^2_2(\nu_t,\mu_t) \leq \E\left[\|X_t - Y_t\|^2\right] \leq r \frac{e^{(4L^2 + 1)t}-1}{4L^2 + 1}.$$
		By Assumption \eqref{ass:wass} we also know that 
		$$W_2(\mu_t,\mu)\leq C_He^{-\kappa \cdot t}W_2(\nu,\mu).$$
		Thus,
		$$W_2(\nu,\mu)\leq  W_2(\nu,\mu_t) + W_2(\mu_t,\mu)\leq \sqrt{r \frac{e^{(4L^2 + 1)t}-1}{4L^2 + 1}} + C_H e^{-\kappa \cdot t}W_2(\nu,\mu),$$
		or equivalently when $t$ is large enough
		$$W_2(\nu,\mu) \leq\sqrt{\frac{r}{4L^2+1}}  \frac{\sqrt{e^{(4L^2 + 1)t}-1}}{1-e^{-\kappa\cdot t}C_H } \leq\sqrt{\frac{r}{4L^2+1}}  \frac{e^{(2L^2 + 1)t}}{1-e^{-\kappa\cdot t}C_H}.$$
		We now take $t = \frac{\ln\left(1+ \frac{2\kappa}{4L^2+1}\right) + \ln \left(C_H\right)}{\kappa}$ to get
		\begin{align*}
		W_2(\nu,\mu) &\leq\sqrt{\frac{r}{4L^2+1}} \left(\frac{4L^2 + 1}{2\kappa}+1\right)\left(1 + \frac{2\kappa}{4L^2 + 1}\right)^\frac{4L^2 +1}{2\kappa}C_H^\frac{4L^2 +1}{2\kappa} \\
		&\leq 10C_H^\frac{4L^2 +1}{2\kappa}\sqrt{r}\left(\frac{L}{\kappa} + 1\right).
		\end{align*}
		To finish the proof it is enough to observe that $r \leq 2\beta^2$.
	\end{proof}
	
\subsection{Proof of Theorem \ref{thm_main_stability}}
To prove Theorem \ref{thm_main_stability}, we will first show that, under suitable assumptions, for a given $t >0$, the measure $\mu_t$ cannot be too different than $\nu_t$. Following \cite{CDL} we define the logarithmic transport distance, which serves as a natural measure of distance between $\mu_t$ and $\nu_t$:
$$\mathcal{D}_{\delta}(\mu, \nu) := \inf_{\pi} \int{\ln\left(1 + \frac{|x-y|^2}{\delta^2}\right)d\pi},$$
where $\delta > 0$ and the infimum is taken over all couplings of $\mu$ and $\nu$, i.e. $\mathcal{D}_\delta$ is a transport cost (but not a distance, and the cost is concave, not convex). 

We have the following connection between $\mathcal{D}_{\delta}$ and $\widetilde{W}_{2,R}^2$, which is essentially the same as \cite[Lemma 5]{Sei17}. The proof of this lemma may be found in the appendix.
\begin{lem} \label{lem_seis_w}
	For any $R,\delta, \epsilon > 0$, we have
	$$\widetilde{W}^2_{2,R}(\mu, \nu) \leq \delta^2\exp\left(\frac{\mathcal{D}_{\delta}(\mu, \nu)}{\epsilon}\right) + R\epsilon + R\frac{\mathcal{D}_{\delta}(\mu, \nu)}{\ln\left(1+\frac{R^2}{\delta^2}\right)}.$$
\end{lem}
\begin{rmq} \label{remark_w1_again}
	We can define $\widetilde{W}_{1,R}$ in the same way, and a similar proof would also show that
	$$\widetilde{W}_{1,R}(\mu, \nu) \leq \delta\exp\left(\frac{\mathcal{D}_{\delta}(\mu, \nu)}{\epsilon}\right) + R\epsilon + R\frac{\mathcal{D}_{\delta}(\mu, \nu)}{\ln\left(1+\frac{R}{\delta^2}\right)},$$
	which motivates Remark \ref{rm_w1_w2}. 
\end{rmq}
Observe that if $\delta < R$, then by choosing $\epsilon = \frac{\mathcal{D}_{\delta}(\mu, \nu)}{\ln\left(1 + \frac{R}{\delta}\right)}$ in the above lemma, we obtain
\begin{equation} \label{eq_log_wass}
\widetilde{W}^2_{2,R}(\mu, \nu) \leq 2R\left(\delta + \frac{\mathcal{D}_{\delta}(\mu, \nu)}{\ln\left(1+\frac{R}{\delta}\right)}\right).
\end{equation}
Moreover, if both $\mu$ and $\nu$ have tame tails then it can be shown that for $R$ large enough,
$$W_2^2(\mu,\nu) \simeq \widetilde{W}^2_{2,R}(\nu, \mu).$$
This is made rigorous in Lemma \ref{lem_truncwass}, in the appendix.
For the logarithmic transport distance, we will prove:
\begin{lem} \label{lem_finite_time}
	Suppose that \eqref{ass:coeff} and \eqref{ass:density} hold and that $X_0 = Y_0$ almost surely with $Y_0 \sim \nu$. Then, for any $t,\delta > 0$,
 	\begin{align*} 
	\mathcal{D}_\delta(\mu_t, \nu_t)
\leq 2t\left(10\left\|d\nu/d\mu\right\|_{L^{p}(\mu)}\|g\|^2_{L^{2q}(\mu)} + \frac{1}{\delta}\|a-b\|_{L^1(\nu)} + \frac{2}{\delta^2}\|\sqrt{\sigma}-\sqrt{\tau}\|^2_{L^2(\nu)}\right).
\end{align*}
where $q$ is such that $\frac{1}{q} + \frac{1}{p} =1$.
\end{lem}
Let us first show how to derive Theorem \ref{thm_main_stability} from Lemma \ref{lem_finite_time}.
\begin{proof}[Proof of Theorem \ref{thm_main_stability}]
	To ease the notation we will denote 
	$$\alpha = 20\left\|d\nu/d\mu\right\|_{L^{p}(\mu)}\|g\|^2_{L^{2q}(\mu)}, \beta = 2\|a-b\|_{L^1(\nu)} + 2\|\sqrt{\sigma}-\sqrt{\tau}\|_{L^2(\nu)}.$$
	We choose $\delta = \beta$ in Lemma \ref{lem_finite_time} and obtain:
	\begin{align*}
	\mathcal{D}_\delta(\mu_t,\nu_t) \leq \left(\alpha + 1\right)t.
	\end{align*}
	Now, combine the above estimate with \eqref{eq_log_wass} to get
	\begin{equation} \label{eq_t_truncated_wass}
	\widetilde{W}^2_{2,R}(\mu_t,\nu_t) \leq 2R\left(\beta + \frac{(\alpha + 1)}{\ln\left(1 + \frac{R}{\beta}\right)}t\right) \leq 2R\frac{(\alpha + 1)}{\ln\left(1 + \frac{R}{\beta}\right)}(t+R).
	\end{equation}
	To see the second inequality note that $\beta \leq R\ln(1+ \frac{R}{\beta})^{-1}.$
	With Assumption \eqref{ass:wass}, we have
	$$\widetilde{W}_{2,R}(\mu_t, \mu) \leq W_2(\mu_t, \mu) \leq C_He^{-t\kappa}W_2(\nu,\mu) \leq C_He^{-t\kappa}\left(\sqrt{\int{|x|^2d\mu} + \int{|x|^2d\nu}}\right).$$
	Observe as well that since $\nu$ is an invariant measure,
	$$\widetilde{W}_{2,R}(\nu_t, \nu) = 0.$$
	We thus get,
	$$\widetilde{W}_{2,R}(\nu_t, \nu) + \widetilde{W}_{2,R}(\mu_t, \mu) \leq C_He^{-t\kappa}\sqrt{m_2^2(\mu)+m_2^2(\nu)}.$$
	Take 
	$$t_0:= \frac{1}{\kappa}\ln\left(\sqrt{m_2^2(\mu)+m_2^2(\nu) \ln\left(1+\frac{R}{\beta}\right)}\right),$$
	for which,
	$$\widetilde{W}_{2,R}(\nu_{t_0}, \nu) + \widetilde{W}_{2,R}(\mu_{t_0}, \mu) \leq \frac{C_H}{\sqrt{\ln\left(1+\frac{R}{\beta}\right)}},$$
	and, by using \eqref{eq_t_truncated_wass},
	$$\widetilde{W}_{2,R}(\mu_{t_0},\nu_{t_0}) \leq \sqrt{\frac{2R(\alpha +1)}{\ln\left(1 + \frac{R}{\beta}\right)}(t_0 + R)}.$$
	To conclude the proof, we use the triangle inequality,
	\begin{align*}
	\widetilde{W}_{2,R}(\mu,\nu)  &\leq \widetilde{W}_{2,R}(\nu_{t_0}, \nu) + \widetilde{W}_{2,R}(\mu_{t_0}, \mu) + \widetilde{W}_2(\mu_{t_0},\nu_{t_0})\\
	&\leq \frac{C_H}{\sqrt{\ln\left(1+\frac{R}{\beta}\right)}} + \sqrt{\frac{2R(\alpha +1)}{\ln\left(1 + \frac{R}{\beta}\right)}(t_0 + R)} \\
	&\leq \frac{1}{\sqrt{\ln\left(1 + \frac{R}{\beta}\right)}}\left(C_H + \sqrt{\frac{2R(\alpha +1)}{\kappa}\ln\left((m_2^2(\mu)+m_2^2(\nu))\ln\left(1+\frac{R}{\beta}\right)\right)+\kappa R}\right).
	\end{align*}
\end{proof}

\subsubsection{Proof of Lemma \ref{lem_finite_time}}
In this section our goal is to bound the logarithmic distance between $X_t$ and $Y_t$ and thus prove Lemma \ref{lem_finite_time}. Towards this, we let $Z_t = X_t - Y_t$. A straightforward application of It\^o's formula gives the following result, whose proof may be found in \cite[Section 4.1]{LL18}.
\begin{lem} \label{lem_li_lou}
\begin{align*}\mathcal{D}_\delta(\mu_t, \nu_t) \leq \mathcal{D}_\delta(\mu_0, \nu_0) &+ 2\int_0^t{\mathbb{E}\left[\frac{\langle Z_s, a(X_s) - b(Y_s)\rangle}{|Z_s|^2 + \delta^2}\right]ds} \\
&+ 2\int_0^t{\mathbb{E}\left[\frac{||\sqrt{\sigma}(X_s) - \sqrt{\tau}(Y_s)||^2}{|Z_s|^2 + \delta^2}\right]ds}.
\end{align*}
\end{lem}
With the above inequality we may then prove.
\begin{lem} \label{lem_li_lou_integral}
	Let $t \geq 0$. Then,
	$$\mathcal{D}_\delta(\mu_t, \nu_t) \leq \mathcal{D}_\delta(\mu_0, \nu_0) + 2\int\limits_0^t\left(5\left(\|g\|^2_{L^2(\mu_s)} + \|g\|^2_{L^2(\nu_s)}\right) + \frac{1}{\delta}\|a-b\|_{L^1(\nu_s)} + \frac{2}{\delta^2}\|\sqrt{\sigma}-\sqrt{\tau}\|^2_{L^2(\nu_s)}\right)ds.$$
\end{lem}
\begin{proof}
We have
\begin{align*}
\mathbb{E}\left[\frac{\langle Z_s, a(X_s) - b(Y_s)\rangle}{|Z_s|^2 + \delta^2}\right] &\leq \mathbb{E}\left[\frac{|a(X_s) - b(Y_s)|}{\sqrt{|Z_s|^2 + \delta^2}}\right] \\
&\leq \mathbb{E}\left[\frac{|a(X_s) - a(Y_s)|}{\sqrt{|Z_s|^2 + \delta^2}}\right] + \mathbb{E}\left[\frac{|a(Y_s) - b(Y_s)|}{\sqrt{|Z_s|^2 + \delta^2}}\right]. \\
\end{align*}
Using Assumption \eqref{ass:coeff}, we get
\begin{align*}
\mathbb{E}\left[\frac{|a(X_s) - a(Y_s)|}{\sqrt{|Z_s|^2 + \delta^2}}\right] &\leq \E\left[g(X_s)+ g(Y_s)\right]=\|g\|_{L^1(\mu_s)} + \|g\|_{L^1(\nu_s)}.\\
\end{align*}
We also have,
\begin{align*}
\mathbb{E}\left[\frac{|a(Y_s) - b(Y_s)|}{\sqrt{|Z_s|^2 + \delta^2}}\right] \leq \frac{1}{\delta}\left\|a-b\right\|_{L^1(\nu_s)}.
\end{align*}
So, 
$$\mathbb{E}\left[\frac{\langle Z_s, a(X_s) - b(Y_s)\rangle}{|Z_s|^2 + \delta^2}\right] \leq \|g\|_{L^1(\mu_s)} + \|g\|_{L^1(\nu_s)} + \frac{1}{\delta}\left\|a-b\right\|_{L^1(\nu_s)}.$$
Similar calculations yield,
$$\mathbb{E}\left[\frac{||\sqrt{\sigma}(X_s) - \sqrt{\tau}(Y_s)||^2}{|Z_s|^2 + \delta^2}\right] \leq 4\|g\|^2_{L^2(\mu_s)} + 4\|g\|^2_{L^2(\nu_s)} + \frac{2}{\delta^2}\left\|\sqrt{\sigma}-\sqrt{\tau}\right\|^2_{L^2(\nu_s)}.$$
As it is fine to assume $\|g\|^2_{L^2(\rho)} \geq \|g\|_{L^1(\rho)} \geq 1$ for any probability measure $\rho$ we consider, since we assumed for convenience that $g \geq 1$, we now plug the above displays into Lemma \ref{lem_li_lou}.
\end{proof}
Lemma \ref{lem_finite_time} is now a consequence of the previous lemma. 
\begin{proof} [Proof of Lemma \ref{lem_finite_time}]
 	We start from Lemma \ref{lem_li_lou_integral}. Since $\mu_0 = \nu_0 = \nu$, and $\nu$ is the invariant measure of the evolution equation for $(Y_t)$, we have
 	\begin{equation} \label{1st_eq_lp_thm}
 	\mathcal{D}_\delta(\mu_t, \nu_t) \leq 2\int\limits_0^t\left(5\left(\|g\|^2_{L^2(\mu_s)} + \|g\|^2_{L^2(\nu)}\right) + \frac{1}{\delta}\|a-b\|_{L^1(\nu)} + \frac{2}{\delta^2}\|\sqrt{\sigma}-\sqrt{\tau}\|^2_{L^2(\nu)}\right)ds.
 	\end{equation}
 	Let $q = \left(1 - \frac{1}{p}\right)^{-1}$. By H\"older's inequality,
 	\begin{align*}
 	\|g\|^2_{L^2(\nu)} &\leq \|g\|^2_{L^{2q}(\mu)}\left\|\frac{d\nu}{d\mu}\right\|_{L^{p}(\mu)}.
 	\end{align*}\
	Also, 
 	 \begin{align*}
 	\|g\|^2_{L^2(\mu_s)} &\leq \|g\|^2_{L^{2q}(\mu)}\left\|\frac{d\mu_s}{d\mu}\right\|_{L^{p}(\mu)} \leq \|g\|^2_{L^{2q}(\mu)}\left\|\frac{d\nu}{d\mu}\right\|_{L^{p}(\mu)},
 	\end{align*}
 	where in the second inequality we have used that $\left\|\frac{d\mu_s}{d\mu}\right\|_{L^{p}(\mu)}$ is monotonic decreasing in $s$.
 	We plug the above displays into \eqref{1st_eq_lp_thm}, to obtain
 	\begin{align*} 
	\mathcal{D}_\delta(\mu_t, \nu_t)
	\leq 2t\left(10\left\|d\nu/d\mu\right\|_{L^{p}(\mu)}\|g\|^2_{L^{2q}(\mu)} + \frac{1}{\delta}\|a-b\|_{L^1(\nu)} + \frac{2}{\delta^2}\|\sqrt{\sigma}-\sqrt{\tau}\|^2_{L^2(\nu)}\right).
	\end{align*}
 	which concludes the proof. 
\end{proof}

\section{Proofs of the applications to Stein kernels} \label{sec_stein}

In this section we fix a measure $\mu$ on $\R^d$, with Stein kernel $\tau_\mu$, constructed as in \eqref{eq_moment_kernel}. For now, we make the assumption that $\tau_\mu$ is positive definite and uniformly bounded from below. To apply our result, we must first construct an It\^o diffusion process with $\mu$ as its unique invariant measure. Define the process $X_t$ to satisfy the following SDE:
\begin{equation} \label{sde_sk}
dX_t = -X_tdt + \sqrt{2\tau_\mu(X_t)}dB_t.
\end{equation}
\begin{lem} \label{lem_invariant_stein}
	$\mu$ is the unique invariant measure of the process $X_t$.
\end{lem}
\begin{proof}
	Let $L$ be the infinitesimal generator of of $(X_t)$. That is, for a twice differentiable test function,
	$$Lf(x) = \langle-x,\nabla f(x)\rangle +\langle \tau_\mu(x),\Hess f(x)\rangle_{HS}.$$
	From the definition of the Stein kernel \eqref{eq_stein_kernel}, we have
	$$\E_\mu\left[L f(x)\right] = 0,$$
	for any such test function. We conclude that $\mu$ is the invariant measure of the process. Uniqueness follows, since $\tau_\mu$ is uniformly bounded from below (\cite{BD17}).
\end{proof}
Before proving Theorem \ref{thm_stein} we collect several facts concerning this process.

\subsection{Lusin-Lipschitz Property for moment maps}
We would now like to claim that the kernel $\tau_\mu$ exhibits Lipschitz-like properties as in Assumption \eqref{ass:coeff}.
For this to hold we restrict our attention to a more regular class of measures. Henceforth, we assume that $\mu = e^{-V(x)}dx$ is an isotropic log-concave measure whose support equals $\R^d$ and that there exists a constant $\alpha > 0$, such that
\begin{equation} \label{eq_bound_condition}
\alpha\mathrm{I_d} \leq \nabla^2V  \leq \frac{1}{\alpha}\mathrm{I_d}.
\end{equation}
In some sense, this assumption can be viewed as restricting ourselves to measures that are not too far from a Gaussian distribution. Under this assumption the main result of this section is that Stein kernels satisfy the Lusin-Lipschitz property that we need in order to apply Theorem \ref{thm_main_stability}. That is: 
\begin{lem} \label{lem_lusin_Lipschitz}
Let $\mu$ be an isotropic log-concave measure on $\R^d$ satisfying \eqref{eq_bound_condition} and let $\tau_\mu$ be its Stein kernel constructed from the moment map. Then, there exists a function $g: \R^d \to \R$ such that for almost every $x,y \in \R^d$:
$$\left\|\sqrt{\tau_\mu(x)} - \sqrt{\tau_\mu(y)}\right\|\leq (g(x)+g(y))\|x-y\|,$$
and, 
$$\|g\|_{L^2(\mu)}\leq Cd^{3/2}\alpha^{-1},$$
where $C > 0$ is a universal constant.
Moreover, there exists a constant $c$ such that if $\mu$ is radially symmetric and has full support then we also have for $d > c$
$$\|g\|_{L^4(\mu)} < Cd^{7/4} \alpha^{-8}.$$
\end{lem}	
In the sequel we will use the following notation, for $v \in \R^d$, $\partial_{v} \varphi$ is the directional derivative of $\varphi$ along $v$. Repeated derivations will be denoted as $\partial^2_{uv}\varphi, \partial^3_{uvw}\varphi,$ etc. If $e_i$, for $i=1,\dots,d$, is a standard unit vector, we will abbreviate $\partial_i \varphi = \partial_{e_i}\varphi$.

Recall that $\tau_\mu = \nabla^2 \varphi (\nabla \varphi ^{-1})$, where $\nabla\varphi$ pushes the measure $e^{-\varphi}dx$ unto $\mu$. Thus, keeping in mind Proposition \ref{prop_RCD}, our first objective is to show $\partial^3_{ijk}\varphi \in W^{1,2}(\mu),$ for every $i,j,k = 1,\dots,d$. This will be a consequence of the following result:

\begin{prop}[Third-order regularity bounds on moment maps]
	\label{prop_3rd_order_bnds}
	Assume that $\mu$ is isotropic and that $\nabla^2 V \geq \alpha\Id$. Then, for $i,j = 1,...,d$ and $j \neq i$, 
	\begin{enumerate}
		\item $\int{|\nabla \partial_{ii}^2 \varphi|^2e^{-\varphi}dx} \leq  C\alpha^{-1}$.
		
		\item $\int{|\nabla \partial_{ij}^2 \varphi|^2e^{-\varphi}dx} \leq C(d+\alpha^{-1}).$
	\end{enumerate}
		Here $C$ is a dimension-free constant, independent of $\mu$. 
\end{prop}
Note that under the isotropy condition, necessarily $\alpha \leq 1$. These bounds build up on the following estimates : 

\begin{prop}  \label{prop_bnds_kk}
Assume that $\mu = e^{-V(x)}dx$ is log-concave and isotropic and $\varphi$ is its moment map. 
\begin{enumerate} 
\item For any direction $e \in \mathbb{S}^{d-1}$ we have
$$\int{(\partial^2_{ee}\varphi)^pe^{-\varphi}dx} \leq 8^pp^{2p}$$

\item $\int{\langle (\nabla^2 \varphi)^{-1}\nabla \partial_{ee} \varphi, \nabla \partial_{ee}\varphi \rangle e^{-\varphi}dx} \leq 32\sqrt{\int{\langle x, e \rangle^4d\mu}} \leq C$, with $C$ a dimension-free constant, that does not depend on $\mu$.  

\item If $\mu$ has a convex support and $\Hess V \geq \alpha \Id$ with $\alpha> 0$, then $\Hess \varphi \leq \alpha^{-1}\Id$. 

\item If $\mu$ has full support and $\Hess V \leq \beta\Id$ with $\beta> 0$ then $\Hess \varphi \geq \beta^{-1}\Id$.
\end{enumerate}
\end{prop}

The first part was proved in \cite{K14} (see \cite[Proposition 3.2]{Fat18} for the precise statement). The second part is an immediate consequence of \cite[eq (55)]{KK15}. The third part was proved in \cite{K14}. The last part is part of the proof of \cite[Theorem 3.4]{KK18}

\begin{proof}[Proof of Proposition \ref{prop_3rd_order_bnds}]
The first part is an immediate consequence of items 2 and 3 of Proposition \ref{prop_bnds_kk}. For the second part, with several successive integrations by parts, we have,
\begin{align} \label{eq_bnd_3rdorder}
\int{(\partial^3_{ijk}\varphi)^2e^{-\varphi}dx} &= -\int{(\partial^4_{iijk}\varphi)(\partial^2_{jk}\varphi)e^{-\varphi}dx} + \int{(\partial^3_{ijk}\varphi)(\partial^2_{jk}\varphi)(\partial_i\varphi)e^{-\varphi}dx} \notag \\
&= \int{(\partial^3_{iik}\varphi)(\partial^3_{jjk}\varphi)e^{-\varphi}dx} - \int{(\partial^3_{iik}\varphi)(\partial^2_{jk}\varphi)(\partial_j\varphi)e^{-\varphi}dx} \notag\\
&\hspace{1cm}+ \int{(\partial^3_{ijk}\varphi)(\partial^2_{jk}\varphi)(\partial_i\varphi)e^{-\varphi}dx} \notag \\
&\leq \frac{1}{2}\int{(\partial^3_{ijk}\varphi)^2e^{-\varphi}dx} + \int{(\partial^3_{iik}\varphi)^2e^{-\varphi}dx} + \frac{1}{2}\int{(\partial^3_{jjk}\varphi)^2e^{-\varphi}dx} \notag \\
&\hspace{1cm} + \frac{1}{2}\int{(\partial^2_{jk}\varphi)^4e^{-\varphi}dx} + \frac{1}{4}\int{((\partial_i\varphi)^4 + (\partial_j\varphi)^4e^{-\varphi}dx}. 
\end{align}

Moreover, since $\nabla^2 \varphi$ is positive-definite, we have $|\partial^2_{jk}\varphi| \leq (\partial^2_{jj}\varphi + \partial^2_{kk} \varphi)/2$, and therefore
\begin{align}
\sum_{k} \int{(\partial^2_{jk}\varphi)^4e^{-\varphi}dx} &\leq \frac{1}{8}\sum_k \int{(\partial^2_{jj}\varphi + \partial^2_{kk} \varphi)^4e^{-\varphi}dx} \notag \\
&\leq Cd.
\end{align}
Summing \eqref{eq_bnd_3rdorder} over $k$ implies the result, via the moment bounds for isotropic log-concave distributions and the 2nd order bounds on $\varphi$. 
\end{proof}
We will also need the following result about radially symmetric functions.
\begin{prop} \label{prop_radial}
	Suppose that \eqref{eq_bound_condition} holds and that $\mu = e^{-V(x)}dx$ is radially symmetric and has full support. Then, there exists an absolute constant $c > 0$, such that if $d > c$:
	$$\int{|\partial_{ijk}^3 \varphi|^4e^{-\varphi}dx} \leq C\alpha^{-30
	}d^4.$$
	for some absolute constant $C > 0$. 
\end{prop}
\begin{proof}
	Note that $\varphi$ satisfies the Monge-Amp\`ere equation
	$$ e^{-\varphi} = e^{-V(\nabla \varphi)}\det\left(\nabla^2\varphi \right),$$
	and that it can be verified that if $V$ is a radial function then so is $\varphi$.
	Let $i = 1,\dots, d$, by taking the logarithm and differentiating the above equation we get:
	\begin{align*}
	\partial_i\varphi = \langle\nabla V(\nabla \varphi),\nabla \partial_i\varphi\rangle - \Tr\left(\nabla^2\partial_i\varphi\left(\nabla^2 \varphi\right)^{-1}\right).
	\end{align*}
	By Proposition \ref{prop_bnds_kk}, $\alpha\mathrm{I_d}\leq \nabla^2\varphi \leq \frac{1}{\alpha}\mathrm{I_d}$. Hence, 
	\begin{equation} \label{eq_trace_bound}
	\Tr\left(\left(\nabla^2 \varphi\right)^{-1}\nabla^2\partial_i\varphi\right) \leq |\partial_i\varphi|+ |\langle\nabla V(\nabla \varphi),\nabla \partial_i\varphi\rangle| \leq |\partial_i\varphi| + \alpha^{-1}\sqrt{d}\|\nabla V(\nabla \varphi)\|,
	\end{equation}
	where the second inequality used Cauchy-Schwartz along with $\|\nabla\partial_i\varphi\|\leq \sqrt{d}\|\nabla^2\varphi\|_{op}$.
	The proof will now be conducted in three steps:
	\begin{enumerate}
		\item We will bound $\Tr\left(\left(\nabla^2 \varphi\right)^{-1}\nabla^2\partial_i\varphi\right)$ in terms of $\Tr\left(\nabla^2\partial_i\varphi\right) = \sum\limits_{j=1}^d\partial_{jji}^3\varphi$.
		\item Using \eqref{eq_trace_bound}, we'll show that $\bigint{ \left(\sum\limits_{j=1}^d\partial_{jji}^3\varphi\right)^4e^{-\varphi}dx}$ cannot be large.
		\item Finally, we will use the previous step to bound  $\bigintss \left(\partial_{kji}^3\varphi\right)^4 e^{-\varphi}dx$.
	\end{enumerate}
	\paragraph{Step 1:}
	We now wish to understand $\Tr\left(\left(\nabla^2 \varphi\right)^{-1}\nabla^2\partial_i\varphi\right)$.
	Write $\varphi(x) = f(\|x\|^2)$, so that,
	\begin{equation} \label{eq_radialHessian}
	\nabla^2\varphi(x) = 2f'(\|x\|^2)\mathrm{I_d} + 4f''(\|x\|^2)xx^T.
	\end{equation}
	The bounds on $\nabla^2\varphi$ imply the following inequalities, which we shall freely use below:
	$$\alpha\leq 2f'(\|x\|), 2f'(\|x\|) + 4f''(\|x\|^2)\|x\|^2 \leq \alpha^{-1},$$
	and
	$$\left|4f''(\|x\|^2)\|x\|^2\right| \leq \alpha^{-1}.$$
	By the Sherman-Morrison formula,
	$$\left(\nabla^2\varphi\right)^{-1}(x) = \frac{1}{2f'(\|x\|^2)}\left(\mathrm{I_d} - \frac{4f''(\|x\|^2)}{2f'(\|x\|^2) +4f''(\|x\|^2)\|x\|^2}xx^T\right).$$
	So,
	\begin{equation}\label{eq_mixed_trace}
	\Tr\left(\left(\nabla^2 \varphi\right)^{-1}\nabla^2\partial_i\varphi\right) = \frac{1}{2f'(\|x\|^2)}\left(\sum\limits_{j=1}^d\partial^3_{jji}\varphi - \frac{4f''(\|x\|^2)}{2f'(\|x\|^2) +4f''(\|x\|^2)\|x\|^2} \partial^3_{xxe_i}\varphi\right).
	\end{equation}
	A calculation shows 
	$$\partial^3_{jji}\varphi(x) = 4x_i(2x_j^2f'''(\|x\|^2) +  f''(\|x\|^2) + 2\delta_{ij}f''(\|x\|^2)),$$
	and
	$$\sum_{j=1}^d\partial^3_{jji}\varphi(x) = 4x_i(2\|x\|^2f'''(\|x\|^2) +  (d+2)f''(\|x\|^2)).$$
	Also,
	$$\partial^3_{xxe_i}\varphi = 4x_i\left(2\|x\|^4f'''(\|x\|^2)+ 3\|x\|^2f''(\|x\|^2)\right).$$
	Thus, if $D(x) := 1- \frac{4f''(\|x\|^2)\|x\|^2}{2f'(\|x\|^2) +4f''(\|x\|^2)\|x\|^2} =  \frac{2f'(\|x\|^2)}{2f'(\|x\|^2) +4f''(\|x\|^2)\|x\|^2}$,
	\begin{align*}
	\sum\limits_{j=1}^d&\partial^3_{jji}\varphi - \frac{4f''(\|x\|^2)}{2f'(\|x\|^2) +4f''(\|x\|^2)\|x\|^2} \partial^3_{xxe_i}\varphi\\
	&= \left(1- \frac{4f''(\|x\|^2)\|x\|^2}{2f'(\|x\|^2) +4f''(\|x\|^2)\|x\|^2}\right)8x_i\|x\|^2f'''(\|x\|^2)\\
	&\ \ \ \  + \left(d+2 - 3\frac{4f''(\|x\|^2)\|x\|^2}{2f'(\|x\|^2)+4f''(\|x\|^2)\|x\|^2}\right)4x_if''(\|x\|^2)\\
	&= D(x)8x_i\|x\|^2f'''(\|x\|^2) + \left(d-1 + 3D(x)\right)4x_if''(\|x\|^2)\\
	&=D(x)\sum\limits_{j=1}^d\partial^3_{jji}\varphi  + \left(d-1 + 3D(x) -D(x)(d+2)\right)4x_if''(\|x\|^2) \\
	&\geq  D(x)\sum\limits_{j=1}^d\partial^3_{jji}\varphi - Cd\alpha^{-3}\frac{|x_i|}{\|x\|^2}\\
	&= D(x)\mathrm{Tr}\left(\nabla^2\partial_i\varphi\right) - Cd\alpha^{-3}\frac{|x_i|}{\|x\|^2}.
	\end{align*}
	In the inequality,  we have used \eqref{eq_radialHessian} along with the bounds $4|f''(\|x\|^2)| \leq \frac{\alpha^{-1}}{\|x\|^2}$, and $\alpha^2 \leq D(x) \leq \alpha^{-2}$. 
	\eqref{eq_mixed_trace} then implies:
	$$\Tr\left(\left(\nabla^2 \varphi\right)^{-1}\nabla^2\partial_i\varphi\right) \geq \frac{D(x)}{2f'(\|x\|^2)}\mathrm{Tr}\left(\nabla^2\partial_i\varphi\right) - \frac{Cd\alpha^{-3}}{2f'(\|x\|^2)}\frac{1}{\|x\|}.$$
	\paragraph{Step 2:}We now integrate with respect to the moment measure, so the estimate from the previous step, along with the bounds $\alpha^2 \leq D(x)$, $\alpha\leq 2f''(\|x\|)\leq\alpha^{-1}$, and \eqref{eq_trace_bound} give:
	\begin{align} \label{eq_moment_laplacian}
	\int\Tr\left(\nabla^2\partial_i\varphi\right)^4e^{-\varphi}dx \leq& C\alpha^{-12}\Big(\int|\partial_i\varphi|^4e^{-\varphi}dx \nonumber\\
	&+ \alpha^{-4}d^2\int\|\nabla V(\nabla \varphi)\|^4e^{-\varphi}dx + d^4\alpha^{-16}\int\frac{1}{\|x\|^{4}}e^{-\varphi}dx\Big).
	\end{align}
	Let us look at each term on the right hand side.
	By a change of variable
	$$\int|\partial_i\varphi|^4e^{-\varphi}dx  =\int\|x_i\|^{4}d\mu \leq C,$$
	since higher moments of coordinates of isotropic log-concave measures are controlled.
	
	For the second term, since $\nabla \varphi$ is a transport map, we get that 
	$$\int\|\nabla V(\nabla \varphi)\|^4e^{-\varphi}dx = \int\|\nabla V\|^4d\mu.$$
	Recalling that $\alpha\mathrm{I_d} \leq \nabla^2 V \leq \frac{1}{\alpha} \mathrm{I_d}$, we apply the Poincar\'e inequality for $\mu$, and since $|\nabla |\nabla V|^2|=2|\nabla^2V\nabla V| \leq 2\alpha^{-1}|\nabla V|$, 
	$$\int\|\nabla V(\nabla \varphi)\|^4e^{-\varphi}dx \leq \left(\int{|\nabla V|^2d\mu}\right)^2 + 4\alpha^{-3}\left(\int{|\nabla V|^2d\mu}\right).$$
	By integration by parts, $\int{|\nabla V|^2d\mu} = \int{\Delta V d\mu} \leq d\alpha^{-1}$, and hence
	$$\int\|\nabla V(\nabla \varphi)\|^4e^{-\varphi}dx \leq 5d^2\alpha^{-4}.$$
	For the third integral, we may use the fact that when $d \geq c$, a reverse H\"older inequality holds for negative moments, and may be applied to radially symmetric log-concave measures (see \cite[Theorem 1.4]{P12}). According to the inequality,
	$$ \int\frac{1}{\|x\|^{4}}e^{-\varphi}dx \leq C\left(\int\|x\|^{2}e^{-\varphi}dx\right)^{-2} \leq C\alpha^{-2}.$$
	We now plug the previous three displays into \eqref{eq_moment_laplacian} to conclude,
	\begin{equation} \label{eq_real_laplacian_bound}
	\bigints\left(\sum\limits_{j=1}^d\partial^3_{jji}\varphi\right)^4e^{-\varphi}dx = \int\mathrm{Tr}\left(\nabla^2\partial_i\varphi\right)^4e^{-\varphi}dx\leq C\alpha^{-30}d^4.
	\end{equation}
	\paragraph{Step 3:} Now, let $i,j,k = 1,\dots,d$ be distinct.
	We have,
	\begin{align*}
	\int\left(\partial^3_{ijk}\varphi\right)^4&e^{-\varphi}dx = 8^4\int\left(x_ix_jx_kf'''(\|x\|^2)\right)^4e^{-\varphi}dx\\
	&\leq \frac{8^4}{2^4}\left(\int x_i^4\left(x^2_jf'''(\|x\|^2)\right)^4e^{-\varphi}dx + \int x_i^4\left(x^2_kf'''(\|x\|^2)\right)^4e^{-\varphi}dx\right).\\
	&\leq \frac{8^4}{2^3}\int x_i^4\left(\|x\|^2f'''(\|x\|^2)\right)^4e^{-\varphi}dx\\
	&\leq 8^4\int x_i^4\left(\|x\|^2f'''(\|x\|^2) + (d+2)f''(\|x\|^2)\right)^4e^{-\varphi}dx \\
	&\ \ \ \ \ \ \ \ \ \ \ + 8^5\int x_i^4\left((d+2)f''(\|x\|^2)\right)^4e^{-\varphi}dx\\
	&=8^4\left(\bigints\left(\sum_{j=1}^d\partial^3_{jji}\varphi(x)\right)^4e^{-\varphi}dx + \int \left((d+2)x_if''(\|x\|^2)\right)^4e^{-\varphi}dx\right)\\
	&\leq\left(C\alpha^{-30}d^4 + (d+2)^4 \int \left(x_if''(\|x\|^2)\right)^4e^{-\varphi}dx\right),
	\end{align*}
	where we have used \eqref{eq_real_laplacian_bound} in the last inequality.  For the remaining integral term, denote $x_{\sim i} = (x_1,...,x_{i-1},x_{i+1},...,x_d)$. Then
	\begin{align*}
	\int \left(x_if''(\|x\|^2)\right)^4e^{-\varphi}dx &= \int \frac{1}{\|x_{\sim i}\|^4}\left(x_i\|x_{\sim i}\|\cdot f''(\|x\|^2)\right)^4e^{-\varphi}dx\\
	&\leq \int \frac{1}{\|x_{\sim i}\|^4}\sum\limits_{j\neq i} x_i^4x_j^4f''(\|x\|^2)^4e^{-\varphi}dx \\
	&=\int \frac{1}{\|x_{\sim i}\|^4}\sum\limits_{j\neq i} \left(\partial^2_{ij}\varphi\right)^4e^{-\varphi}dx\\
	&\leq d\alpha^{-4}\int \frac{1}{\|x_{\sim i}\|^4}e^{-\varphi}dx \leq Cd\alpha^{-6}.
	\end{align*}
	In the last inequality we again used the a reverse H\"older inequality for negative moments of radially symmetric log-concave measures.
	Plugging this estimate into the previous display finishes the proof, when  $|\{i,j,k\}| = 3$. The other cases can be proven similarly.
\end{proof}
We are now in a position to prove Lemma \ref{lem_lusin_Lipschitz}.
\begin{proof}[Proof of Lemma \ref{lem_lusin_Lipschitz}]
	Since $\nabla \varphi ^{-1}$ transports $\mu$ to $e^{-\varphi}dx$, from Proposition \ref{prop_3rd_order_bnds} we conclude that $\tau_\mu \in W^{1,2}(\mu)$ and that $\|D\tau_\mu\|_{L^2(\mu)} \leq Cd^{3/2}\alpha^{-1/2}$, where $D$ stands for the total derivative operator. Thus, by Proposition \ref{prop_RCD}, there exists a function $\tilde{g}$ for which,
	$$\|\tau_\mu(x)-\tau_\mu(y)\|\leq (\tilde{g}(x)+\tilde{g}(y))\|x-y\|,$$
	and $\|\tilde{g}\|_{L^2(\mu)} \leq C\frac{d^{3/2}}{\alpha^{1/2}}$.
	Proposition \ref{prop_bnds_kk} along with \eqref{eq_bound_condition} shows
	$$\sqrt{\alpha}\Id \leq \sqrt{\tau_\mu}.$$
	Hence,
	$$\left\|\sqrt{\tau_\mu(x)} - \sqrt{\tau_\mu(y)}\right\| \leq \frac{1}{\sqrt{\alpha}}\|\tau_\mu(x)-\tau_\mu(y)\|.$$
	Take now $g := \frac{1}{\sqrt{\alpha}}\tilde{g}$ to conclude the proof. 
	If $\mu$ is radially symmetric, then Proposition \ref{prop_radial} shows $\|D\tau_\mu\|_{L^4(\mu)}\leq Cd^{\frac{7}{4}} \alpha^{-\frac{15}{2}}$ and the proof continues in a similar way.
\end{proof}
\subsection{Exponential convergence to equilibrium}
We now show that the process \eqref{sde_sk} satisfies the exponential convergence to equilibrium property we require, as long as \eqref{eq_bound_condition} is satisfied.
\begin{lem} \label{lem_stein_contraction}
Assume that $\mu = e^{-V}dx$ with $\alpha^{-1}\Id \geq \nabla^2 V \geq \alpha\Id$. Then the diffusion process \eqref{sde_sk} satisfies Assumption \ref{ass:wass} with $\kappa = 1/2$ and $C_H = \alpha^{-2}$. 
\end{lem}

\begin{proof}
As demonstrated in \cite{Kol14}, the diffusion process $\nabla \varphi^*(X_t)$, where  $(X_t)$ solves \eqref{sde_sk}, satisfies the Bakry-Emery curvature dimension condition CD($1/2, \infty$) when viewed as the canonical diffusion process on the weighted manifold $(\R^d, (\nabla^2 \varphi)^{-1}, e^{-\varphi})$. Therefore it is a contraction in Wasserstein distance, with respect to the Riemannian metric $d_{\varphi}$ with tensor $(\nabla^2 \varphi)^{-1}$ \cite{vRS05}. That is
$$W_{2, d_{\varphi}}(\mu_t \circ \nabla \varphi, e^{-\varphi}) \leq e^{-t/2}W_{2, d_{\varphi}}(\mu_0 \circ \nabla \varphi, e^{-\varphi}).$$
From the bounds on $\nabla^2 \varphi$ given by Proposition \ref{prop_bnds_kk}, we have
$$\alpha^{-1}|x-y|^2 \geq d_{\varphi}(x,y)^2 \geq \alpha|x-y|^2,$$
and the result follows, using again the two-sided Lipschitz bounds on $\nabla \varphi$. 
\end{proof}

\subsection{Stability for Stein kernels}
\begin{proof}[Proof of Theorem \ref{thm_stein}]
	We consider the two processes 
	\begin{align*}
	dX_t &=  -X_t + \sqrt{2\tau(X_t)}dB_t,\\
	dY_t &=  -Y_t + \sqrt{2\sigma(Y_t)}dB_t.\\
	\end{align*}
	By Lemma \ref{lem_invariant_stein}, $\mu$ and $\nu$ are the respective invariant measures of $X_t$ and $Y_t$.
	By Lemma \ref{lem_stein_contraction}, Assumption \ref{ass:wass} is satisfied with $\kappa = \frac{1}{2}, C_H = \alpha^{-2}$.
	By Lemma \ref{lem_lusin_Lipschitz}, Assumption \ref{ass:coeff} is satisfied with $\|g\|^2_{L^2(\mu)} \leq C d^3\alpha^{-2}$.
	
	Set $p = \infty, q = 1$ and $R > 0$. Plugging the above estimates to Theorem \ref{thm_main_stability}, we get
	$$\widetilde{W}_{2,R}^2(\nu,\mu) \leq C\alpha^{-6}d^3R\left\|d\nu/d\mu\right\|_{\infty}\frac{\ln\left(\ln\left(1 + \frac{R}{\beta}\right)\right) +\ln\left(M\right)+R}{\ln\left(1 + \frac{R}{\beta}\right)}.$$
	$\nu$ and $\mu$ are log-concave and in-particular have sub-exponential tails. 
	We apply Lemma \ref{lem_truncwass}, from the appendix, to obtain 
	a constant $C' > 0$ such that 
	$$W_2(\nu,\mu) \leq 2\widetilde{W}_{2,C'M\ln(M)}(\nu,\mu),$$ which proves the first part of the theorem. For the second part, if $\mu$ is radially symmetric, then we take $p = q = 2$, and by Lemma \ref{lem_lusin_Lipschitz}, \ref{ass:coeff} is now satisfied with $\|g\|^2_{L^4(\mu)} \leq C d^{\frac{7}{2}}\alpha^{-16}$ and the proof continues in the same way.
	The last part of the theorem is an immediate consequence of Theorem \ref{thm_stability_lip}.
\end{proof}

\paragraph{\underline{Acknowledgments:}} We thank Ronen Eldan, Bo'az Klartag, Claude Le Bris, Michel Ledoux and Guillaume Mijoule for useful discussions and for their enlightening comments. Part of this work was done while M.F. was a guest of the teams Matherials and Mokaplan at the INRIA Paris, whose hospitality is gratefully acknowledged. M.F. was supported by the Projects MESA (ANR-18-CE40-006) and EFI (ANR-17-CE40-0030) of the French National Research Agency (ANR).

\newpage
\appendix
\section{Transport inequalities for the truncated Wasserstein distance}
\begin{proof}[Proof of Lemma \ref{lem_seis_w}]
		We let $\pi$ denote the optimal coupling for $\mathcal{D}_{\delta}$ and $(X,Y)\sim \pi$. Define the sets
		\begin{align*}
		D &= \left\{(x,y) \in \R^{2d} : \|x-y\|^2 \leq R\right\},\\
		\widetilde{D} &= \left\{(x,y) \in D : \ln\left(1 + \frac{\|x-y\|^2}{\delta^2} \right)\leq \frac{\mathcal{D}_{\delta}(\mu, \nu)}{\epsilon}\right\}.
		\end{align*}
		We now write
		$$\E\left[\min\left(\|X-Y\|^2,R\right)\right] = \E\left[\|X-Y\|^2\mathbbm{1}_{\widetilde{D}} \right] + \E\left[\|X-Y\|^2\mathbbm{1}_{D\setminus\widetilde{D}}\right] + R\cdot\E\left[\mathbbm{1}_{\R{^{2d}}}\setminus D\right],$$
		and bound each term separately.
		Observe that for any $\alpha \in \R$,
		$$ \ln\left(1+\frac{x^2}{\delta^2}\right) \leq \alpha \iff |x| \leq  \delta\sqrt{e^\alpha-1}.$$
		Thus,
		$$\E\left[\|X-Y\|^2\mathbbm{1}_{\widetilde{D}} \right] \leq \delta^2 \exp\left(\frac{\mathcal{D}_{\delta}(\mu, \nu)}{\epsilon}\right).$$
		Next, by Markov's inequality
		$$\PP\left(\ln\left(1 + \frac{\|X-Y\|^2}{\delta^2}\right) \geq \frac{\mathcal{D}_{\delta}(\mu, \nu)}{\epsilon}\right) \leq \frac{\epsilon}{\mathcal{D}_{\delta}(\mu, \nu)}\E\left[\ln\left(1 + \frac{\|X-Y\|^2}{\delta^2}\right)\right] = \epsilon.$$
		So,
		$$\E\left[\|X-Y\|^2\mathbbm{1}_{D\setminus\widetilde{D}}\right] \leq R\E\left[\mathbbm{1}_{\R^{2d}\setminus\widetilde{D}}\right] \leq R\epsilon.$$
		Finally, a second application of Markov's inequality gives
		\begin{align*}
		R\E\left[\mathbbm{1}_{\R{^{2d}}}\setminus D\right] &\leq R \cdot \PP\left(\|X-Y\|^2 \geq R\right) \\
		&=R\cdot\PP\left(\ln\left(1+\frac{\|X-Y\|^2}{\delta^2}\right) \geq \ln\left(1 + \frac{R^2}{\delta^2}\right)\right) \leq R \frac{\mathcal{D}_{\delta}(\mu, \nu)}{\ln\left(1+\frac{R^2}{\delta^2}\right)}.
		\end{align*}
\end{proof}
\begin{lem} \label{lem_truncwass}
	Let $X\sim \mu,Y \sim \nu$ be two centered random vectors in $\mathbb{R}^d$. Assume that both $X$ and $Y$ are sub-exponential with parameter $M$, in the sense that for every $k \geq 2$,
	$$\frac{\mathbb{E}\left[\|X\|^k\right]^{\frac{1}{k}}}{k},\frac{\mathbb{E}\left[\|Y\|^k\right]^{\frac{1}{k}}}{k} \leq M.$$
	Then
	$$W_2^2(\mu,\nu) \leq 2\widetilde{W}_{2,CM\ln(M)}^2(\mu,\nu),$$
	for a universal constant $C > 0$.
\end{lem}
\begin{proof}
	Fix $R > 0$ and let $\pi$ denote the optimal coupling for $\widetilde{W}_{2,R}$ and $(X,Y)\sim \pi$.
	Then,
	\begin{align*}
	 \mathbb{E}\left[\|X-Y\|^2\right]&= \mathbb{E}\left[\|X-Y\|^2\mathbbm{1}_{\|X-Y\|^2\leq R}\right] +\mathbb{E}\left[\|X-Y\|^2\mathbbm{1}_{\|X-Y\|^2>R}\right]\\
	&\leq \widetilde{W}^2_{2,R}(\mu,\nu) + \sqrt{\mathbb{E}\left[\|X-Y\|^4\right]\PP\left(\|X-Y\|>R\right)}.
	\end{align*}
	$X-Y$ also has a sub-exponential law with parameter $C'M$, where $C'>0$ is a constant. Thus, for some other constant $C$,
	$$\sqrt{\mathbb{E}\left[\|X-Y\|^4\right]} \leq CM\mathbb{E}\left[\|X-Y\|^2\right],$$
	and
	$$\sqrt{\PP\left(\|X-Y\|>R\right)} \leq e^{-\frac{R}{CM}}.$$
	Take $R = CM\ln(2CM)$, to get.
	$$\mathbb{E}\left[\|X-Y\|^2\right] \leq \widetilde{W}^2_{2,R}(\mu,\nu) + \frac{1}{2}\mathbb{E}\left[\|X-Y\|^2\right],$$
	which implies,
	$$W_2^2(\mu,\nu)\leq \mathbb{E}\left[\|X-Y\|^2\right] \leq 2\widetilde{W}^2_{2,R}(\mu,\nu).$$
\end{proof}

\end{document}